\documentclass[11pt]{article}
\usepackage{amsmath}
\usepackage{amssymb,amsbsy,amsthm}
\usepackage[dvips]{graphicx}
\usepackage[dvipsnames]{xcolor}
\usepackage{pstricks}
\usepackage{pst-eps}
\usepackage{pst-node}
\usepackage{pst-all}
\usepackage{enumerate}
\usepackage[margin=1in]{geometry}
\usepackage{hyperref}

\allowdisplaybreaks[1]

\newcommand{\p}{\mathbb{P}} 
\newcommand{\E}{\mathbb{E}} 
\newcommand{\eps}{\varepsilon} 
\DeclareMathOperator{\Var}{Var} 

\newcommand{\wh}{\widehat}  
\newcommand{\wt}{\widetilde} 

\def\cE{{\mathcal E}}

\def\cS{{\mathcal S}}

\newtheorem{theorem}{Theorem}
\newtheorem{lemma}[theorem]{Lemma}

\newtheorem{corollary}[theorem]{Corollary}

\begin{document}

\title{Finding a planted clique by adaptive probing}
\author{
        Mikl\'os Z.\ R\'acz 
        \thanks{Princeton University; \texttt{mracz@princeton.edu}. Research supported in part by NSF grant DMS 1811724.}
        \and         
		Benjamin Schiffer
        \thanks{Princeton University \texttt{bgs3@princeton.edu}.}
}
\date{\today}

\maketitle


\begin{abstract}
We consider a variant of the planted clique problem where we are allowed unbounded computational time but can only investigate a small part of the graph by adaptive edge queries. We determine (up to logarithmic factors) the number of queries necessary both for detecting the presence of a planted clique and for finding the planted clique. 

Specifically, let $G \sim G(n,1/2,k)$ be a random graph on $n$ vertices with a planted clique of size $k$. We show that no algorithm that makes at most $q = o(n^2 / k^2 + n)$ adaptive queries to the adjacency matrix of $G$ is likely to find the planted clique. On the other hand, when $k \geq (2+\eps) \log_2 n$ there exists a simple algorithm (with unbounded computational power) that finds the planted clique with high probability by making $q = O( (n^2 / k^2) \log^2 n + n \log n)$ adaptive queries. For detection, the additive $n$ term is not necessary: the number of queries needed to detect the presence of a planted clique is $n^2 / k^2$ (up to logarithmic factors). 
\end{abstract}


\section{Introduction} \label{sec:intro} 

In the planted clique problem the goal is to find a clique that is planted within an Erd\H{o}s-R\'enyi random graph. 
This problem has received widespread attention in the past few decades 
because there exists a (wide) range of clique sizes for which 
it is information-theoretically possible to find the planted clique 
but there are no known polynomial-time algorithms to do so~\cite{jerrum1992large,kuvcera1995expected,AKS98,feige2010finding,DGGP14,DM15}. 
In this regime it is conjectured to be computationally hard to find the planted clique 
and this conjecture forms the basis of numerous average-case complexity results in recent years~\cite{berthet2013optimal,berthet2013complexity,gao2017sparse,brennan2018reducibility,brennan2019optimal}.

In this paper we consider a variant of the planted clique problem where we are allowed unbounded computational time but can only investigate a small part of the graph by adaptive edge queries. 
We consider the problems of detection and estimation under this model, and 
determine (up to logarithmic factors) the number of queries necessary both for detecting the presence of a planted clique and for finding the planted clique.

In the problems we consider there is an underlying $n$ vertex graph $G$ with vertex set $\left[ n \right] := \left\{ 1, 2, \ldots, n \right\}$. 
The algorithms that we consider are allowed unbounded computational power but we restrict the number of edges they are allowed to inspect. 
Specifically, we consider algorithms that evolve dynamically over a certain number of steps. 
In the first step, the algorithm chooses a pair $(i_1, j_1)$, $1 \leq i_1 < j_1 \leq n$, and asks whether this pair is an edge or not. 
Depending on the outcome, the algorithm selects a second pair $(i_2, j_2)$, $1 \leq i_2 < j_2 \leq n$, and asks whether this pair is an edge or not. 
It then selects $(i_3, j_3)$, and so on. 
The algorithm may ask $q$ such edge queries and use unbounded computational time to produce an output.

The detection problem can be phrased as a simple hypothesis testing problem. 
Under the null hypothesis $H_{0}$, the graph $G$ is an Erd\H{o}s-R\'enyi random graph with edge density $1/2$. 
Under the alternative hypothesis $H_{1}$, the graph $G$ is drawn from the planted clique model with clique size $k$. 
That is, we first choose a (uniformly) random subset of the vertices $K \subseteq \left[ n \right]$ of size $\left| K \right| = k$, we connect all pairs of vertices in $K$---that is, the vertices in $K$ form a clique---and every other pair of vertices is connected independently with probability $1/2$. 
In short:  
\begin{equation}\label{eq:hypothesis}
 H_{0} : G \sim G(n,1/2), 
 \qquad \qquad
 H_{1} : G \sim G(n,1/2,k).
\end{equation}
We denote the two probability distributions over $n$ vertex graphs by $\p_{0}$ and $\p_{1}$, respectively. 
An algorithm $A$ for detection under the adaptive edge query model makes up to $q$ adaptive edge queries to $G$ and then outputs a hypothesis in $\left\{ 0, 1 \right\}$. 
We measure the performance of an algorithm $A$ by its risk, which is defined as the sum of its type I and type II errors: 
\[
 R(A) := \p_{0} \left( A(G) = 1 \right) + \p_{1} \left( A(G) = 0 \right). 
\]
If an algorithm $A$ achieves vanishing risk---$R(A) = o(1)$ as $n \to \infty$---then we say that $A$ can detect the presence of a planted clique; otherwise, we say that it cannot do so. 

The following theorem determines (up to logarithmic factors) the number of queries necessary to detect the presence of a planted clique. All logarithms in this paper are in base~2. 
\begin{theorem}[Detecting a planted clique]\label{thm:detection}
 Consider the hypothesis testing problem in~\eqref{eq:hypothesis}. 
 \begin{enumerate}[(a)]
  \item\label{thm:detection_lb} Let $q = o(n^2 / k^2)$ as $n \to \infty$. 
  If an algorithm $A$ makes at most $q$ adaptive edge queries then its risk must satisfy $R(A) \geq 1 - o(1)$ as $n \to \infty$. 
  \item\label{thm:detection_ub} Suppose that $k \geq (2+\eps) \log n$ for some constant $\eps > 0$ and let $\eps_{0} > 0$ be arbitrary. 
  There exists an algorithm (with unlimited computational power) that can detect the presence of a planted clique by querying 
  \[
   q = (2+\eps_{0}) \frac{n^{2}}{k^{2}} \log^{2} n
  \]
  pairs of vertices. Moreover, the queries can be nonadaptive. 
 \end{enumerate}
\end{theorem}

If we can detect the presence of a planted clique, the natural next goal is to find it. 
The following theorem determines (up to logarithmic factors) the number of queries necessary to find the planted clique. 
In particular, it shows that an extra $n \log n$ queries suffice compared to detection and that this is tight (up to logarithmic factors). 

\begin{theorem}[Finding the planted clique]\label{thm:estimation}
 Let $G \sim G(n,1/2,k)$, where $1 \leq k < n$. 
 \begin{enumerate}[(a)]
  \item\label{thm:estimation_lb} Let $q = o(n^2 / k^2 + n)$ as $n \to \infty$.   
  No algorithm that makes at most $q$ adaptive edge queries can find the planted clique. 
  That is, any estimator $\wh{K}$ of the planted clique $K$ that is based on at most $q$ adaptive edge queries satisfies 
  $\p_{1} \left( \wh{K} = K \right) = o(1)$ as $n\to \infty$. 
  \item\label{thm:estimation_ub} Suppose that $k \geq (2+\eps) \log n$ for some constant $\eps > 0$ and let $\eps_{0} > 0$ be arbitrary. 
  There exists an algorithm (with unlimited computational power) that adaptively queries 
  \[
   q = (2+\eps_{0}) \frac{n^{2}}{k^{2}} \log^{2} n + (1+\eps_{0}) n \log n 
  \]
  pairs of vertices 
  and finds the planted clique with probability $1 - o(1)$ as $n\to\infty$. 
 \end{enumerate}
\end{theorem}

Theorems~\ref{thm:detection} and~\ref{thm:estimation} give a complete phase diagram of when detection and estimation are possible as a function of the clique size $k$ and the number of queries $q$ (up to some boundary cases). 
A natural parametrization is to take both $k$ and $q$ to be polynomial in $n$: 
$k = n^{\gamma}$ 
and 
$q = n^{\delta}$ 
for some $\gamma \in (0,1)$ and $\delta \in (0,2)$. 
Corollary~\ref{cor:phase_diagram} summarizes the results with this parametrization---see also Figure~\ref{fig:phase_diagram} for an illustration. 
Note in particular the region of the phase space where detection is possible but estimation is not. 
Note also that the conjectured computational threshold is at~$\gamma = 1/2$. 

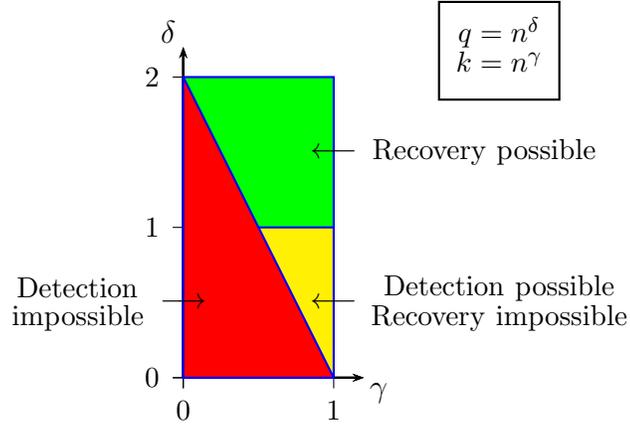
\begin{figure}[t!]\label{fig:phase_diagram}
 \centering 
   \psset{xunit=2cm,yunit=2cm}
   \begin{pspicture}(-0.5,-0.3)(2,2.4)
    \psaxes[labelFontSize=\textstyle]{->}(0,0)(0,0)(1.2,2.2)
    \pspolygon*[linecolor=red](0,2)(1,0)(0,0)
    \pspolygon*[linecolor=green](0,2)(0.5,1)(1,1)(1,2)
    \pspolygon*[linecolor=yellow](0.5,1)(1,0)(1,1)
    \pspolygon[linecolor=blue](0,0)(0,2)(1,2)(1,0)
    \psline[linecolor=blue](0,2)(1,0)
    \psline[linecolor=blue](0.5,1)(1,1)
    \rput(-0.1,2.3){\large $\delta$}
    \rput(2.0,1.5){Recovery possible}
    \rput(1.0,1.5){$\longleftarrow$}
    \rput(2.1,0.6){Detection possible}
    \rput(2.1,0.4){Recovery impossible}
    \rput(1.0,0.5){$\longleftarrow$}
    \rput(2.1,2.3){$q=n^\delta$}
    \rput(2.1,2.1){$k = n^\gamma$}
    \pspolygon[linecolor=black](1.7,1.85)(1.7,2.5)(2.5,2.5)(2.5,1.85)
    \rput(-0.7,0.6){Detection}
    \rput(-0.7,0.4){impossible}
    \rput(-0.0,0.5){$\longrightarrow$}
    \rput(1.3,-0.1){\large $\gamma$}
   \end{pspicture}
 \caption{Phase diagram for detecting the presence of a planted clique and for finding the planted clique, 
 as a function of the clique size $k = n^{\gamma}$ and the number of adaptive edge queries $q = n^{\delta}$. The horizontal axis contains $\gamma \in (0,1)$, the vertical axis contains $\delta \in (0,2)$.}
\end{figure}

\begin{corollary}[Phase diagram]\label{cor:phase_diagram}
 Suppose that 
 $k = n^{\gamma}$ 
 and 
 $q = n^{\delta}$ 
 for some $\gamma \in (0,1)$ and $\delta \in (0,2)$. 
 \begin{enumerate}[(a)]
  \item\label{cor:impossible} 
  If $\delta < 2 - 2 \gamma$, 
  then detecting the presence of a planted clique is impossible. 
  \item\label{cor:detection_possible} 
  If $\delta > 2 - 2 \gamma$ and $\delta < 1$, 
  then it is possible to detect the presence of a planted clique, 
  but it is impossible to find the planted clique. 
  \item\label{cor:finding_possible} 
  If $\delta > 2 - 2 \gamma$ and $\delta > 1$, 
  then it is possible to find the planted clique. 
 \end{enumerate}
\end{corollary}

These results raise several open questions. 
First, can the logarithmic factors be closed, to obtain results that are tight up to constant factors? 
Second, how do these results change if we plant a different subgraph instead of a clique? 
For instance, one can plant a dense random graph $G(k,q)$ with $q > 1/2$. 
Finally, while we neglected all computational considerations in this paper, are there any connections to average-case computational hardness? 
We leave exploring these questions to future work.

The rest of this paper is outlined as follows. 
After discussing motivation and related work in the remainder of the introduction, 
we turn to algorithms for detection and estimation in Section~\ref{sec:algorithms}, 
proving Theorem~\ref{thm:detection}\eqref{thm:detection_ub} and Theorem~\ref{thm:estimation}\eqref{thm:estimation_ub}. 
Finally, we prove the impossibility results of Theorem~\ref{thm:detection}\eqref{thm:detection_lb} and Theorem~\ref{thm:estimation}\eqref{thm:estimation_lb} in Section~\ref{sec:lb}.

\subsection{Motivation} \label{sec:motivation} 

There are several potential applications where understanding the query complexity necessary to finding cliques may be of interest. 
For instance, in scientific applications one may wish to find closely related entities (corresponding to a clique or dense subgraph), and querying an edge may correspond to performing a physical experiment which is costly and/or time-consuming.

Another potential application is to the analysis of social media connections. Here the nodes of a graph represent individuals and the edges represent connections between individuals such as Facebook friends, Twitter following, or LinkedIn connections. Access to these connections may be expensive to obtain or limited (due to privacy limitations or any other source of incomplete information), and hence query complexity may be relevant when trying to reconstruct a specific close-knit group within the network. 

The planted clique problem and related subgraph inference problems have been applied to a variety of applications, including biological networks~\cite{bio}, cryptography~\cite{crypto}, and finance~\cite{finance}. 
Obtaining full information about the underlying networks in these applications may not be possible due to queries being expensive and/or limited, and hence the planted clique problem with limited adaptive probing could be relevant to these same applications.

\subsection{Related work} \label{sec:related_work} 

This paper is a natural follow-up to the recent work of Feige, Gamarnik, Neeman, R\'acz, and Tetali~\cite{FGNRT18}, 
where the authors consider the problem of finding cliques in an Erd\H{o}s-R\'enyi random graph under the same adaptive edge query model. 
While the largest clique in an Erd\H{o}s-R\'enyi random graph with edge density $1/2$ has size approximately $2 \log n$, 
the current best algorithm that makes at most $q = O(n^{\delta})$ adaptive edge queries 
finds a clique of size approximately $(1+\delta/2) \log n$. 
Feige~et~al.~\cite{FGNRT18} show an impossibility result if the adaptivity of the algorithm is limited: 
any algorithm that makes $q = O(n^{\delta})$ edge queries ($\delta < 2$) in $\ell$ rounds 
finds cliques of size at most $(2-\eps) \log n$ where $\eps = \eps \left( \delta, \ell \right) > 0$. 
Very recently, Alweiss, Ben Hamida, He, and Moreira~\cite{alweiss2019subgraph} improved upon this result, showing that there exists such $\eps$ that depends only on $\delta$ and not on $\ell$. 
However, closing the gap between the upper and lower bounds remains an open problem.

Several recent works consider finding structure in a random graph under such an adaptive edge query model. 
Ferber, Krivelevich, Sudakov, and Vieira studied finding a Hamilton cycle~\cite{ferber2016finding} and finding long paths~\cite{ferber2017finding}, while Conlon, Fox, Grinshpun, and He~\cite{conlon2019online} studied finding a copy of a fixed target graph (such as a constant size clique). All of these works focus on sparse random graphs.

As mentioned in the introduction, the planted clique problem has been studied from many angles in the past few decades~\cite{jerrum1992large,kuvcera1995expected,AKS98,feige2010finding,DGGP14,DM15,berthet2013optimal,berthet2013complexity,gao2017sparse,brennan2018reducibility,brennan2019optimal}. 
To the best of our knowledge, it has not been considered under an edge query model before. 
It would be interesting to see if there are any connections to computational aspects of the planted clique problem. 
The recent work of Mardia, Ali, and Chandrasekher~\cite{mardia2020finding} develops sublinear\footnote{Here the input size is $\Theta(n^{2})$, hence sublinear refers to $o(n^{2})$.} time algorithms for finding the planted clique in the regime $k = \omega \left( \sqrt{n \log \log n} \right)$ and makes such connections. As the authors point out, our results imply an $\Omega \left( n \right)$ running time lower bound for finding the planted clique, which shows that their algorithms are optimal (up to logarithmic factors) whenever $k = \Omega \left( n^{2/3} \right)$.

Finally, we mention that query complexity 
arises naturally in many other areas, such as clustering~\cite{vinayak2016crowdsourced,mazumdar2017clustering,mazumdar2017query}, where answers to queries are often noisy due to them being crowdsourced, 
and community detection~\cite{hartmann2016clustering,anagnostopoulos2016community}, where the evolution of the underlying graph necessitates repeated queries. 
Statistical queries have also been widely studied~\cite{kearns1998efficient}, including in the setting of the planted clique model~\cite{feldman2017statistical}. 
More generally, our work fits into the framework of online learning, a large and rapidly growing area which is beyond the scope of this article to survey.

\section{Algorithms} \label{sec:algorithms} 

We start with a simple sampling-based algorithm to detect the presence of a planted clique. 
This is contained in Section~\ref{sec:alg_detection} and proves Theorem~\ref{thm:detection}\eqref{thm:detection_ub}. 
We then extend this algorithm in Section~\ref{sec:alg_finding} to find all vertices of the planted clique, thus proving Theorem~\ref{thm:estimation}\eqref{thm:estimation_ub}. 

First, recall that the largest clique in an Erd\H{o}s-R\'enyi random graph has size approximately $2\log n$. 
In fact, very precise results are known. 
Define 
$\omega_{n} = 2 \log n - 2 \log\log n + 2 \log e - 1$. 
Matula~\cite{matula1972} showed that for any $\varepsilon >0$, the clique number (the size of the largest clique) $\omega (G)$ of a random graph $G$ drawn from $G(n,1/2)$ satisfies $\left\lfloor \omega_{n} - \varepsilon \right\rfloor \leq \omega (G) \leq \left\lfloor \omega_{n} + \varepsilon \right\rfloor$ with probability tending to $1$ as $n \to \infty$; see also~\cite{bollobas1976cliques}. 
For our purposes much weaker estimates suffice. 
Indeed, a first moment argument shows that $\p_{0} \left( \omega(G) \leq 2\log n + 3 \right) \to 1$ as $n \to \infty$ (see~\cite{lugosi17}).

\subsection{Detecting the presence of a planted clique} \label{sec:alg_detection} 

The basic idea in detecting the presence of a planted clique is to sample all pairs of vertices among a set $S \subseteq [n]$ of size $m := \left( 2 + \eps' \right) (n/k) \log n$. After these queries we learn the induced subgraph on~$S$ and we can use the size of the largest clique in $S$ as a statistic to distinguish between the hypotheses $H_{0}$ and $H_{1}$, as follows. We know (see above) that under $\p_{0}$ this statistic is at most $2 \log n + 3$ with probability $1 - o(1)$. On the other hand, under $\p_{1}$ the set $S$ contains, in expectation, $\left( 2 + \eps' \right) \log n$ vertices from the planted clique, 
so with probability $1-o(1)$ this statistic is at least $\left( 2 + \eps' / 2 \right) \log n$.

The following proof makes this reasoning precise. 
Note that this algorithm to detect the presence of a planted clique is nonadaptive, making all queries at the same time. 

\begin{proof}[Proof of Theorem~\ref{thm:detection}\eqref{thm:detection_ub}] 
Let $\eps' > 0$ be such that $2 \eps' + \eps'^{2}/2 \leq \eps_{0}$ and $\eps' \leq \eps$. 
First, we choose an arbitrary subset $S$ of the $n$ vertices of size 
$m := |S| = \left( 2 + \eps' \right) (n/k) \log n$; 
for instance, choose $S:= \left\{ 1, 2, \ldots, m \right\}$. 
We then query all pairs of vertices among $S$. 
This results in 
\[
\binom{m}{2} \leq \frac{m^{2}}{2} = (2 + 2\eps' + \eps'^{2} / 2) \frac{n^{2}}{k^{2}} \log^{2} n  
\leq (2 + \eps_{0}) \frac{n^{2}}{k^{2}} \log^{2} (n)
\]
queries. 
After the queries we know the induced subgraph on~$S$. 
In particular---due to the fact that we have no restrictions on computational power---we can compute the size of the largest clique in this induced subgraph. 
The algorithm then chooses a hypothesis based on this statistic: 
if~$S$ contains a clique of size at least 
$\left( 2 + \eps'/2 \right) \log n$, 
then it accepts the alternative hypothesis $H_{1}$; 
otherwise, it accepts the null hypothesis $H_{0}$.

We now argue that this algorithm achieves vanishing risk. First, as we discussed at the beginning of Section~\ref{sec:algorithms}, 
if $G \sim G \left( n, 1/2 \right)$, then the largest clique in $G$ has size approximately $2 \log n$. In particular, 
$\p_{0} \left( \omega (G) \geq \left( 2+ \eps'/2 \right) \log n \right) \to 0$ as $n \to \infty$; that is, the type I error vanishes in the limit. 

Next, assuming $H_{1}$, let $X$ denote the number of planted clique vertices in $S$. 
Observe that $X$ has a hypergeometric distribution with parameters $n$, $k$, and $m$. 
Thus we have that $\E \left[ X \right] = m \times (k/n) = (2+\eps') \log(n)$ 
and $\Var(X) \leq m \frac{k}{n}(1-\frac{k}{n}) \leq (2+\eps') \log(n)$, 
so Chebyshev's inequality implies that 
$\p_{1} \left( X \leq (2+\eps'/2) \log(n) \right) \leq c / \log(n)$ 
for some constant $c < \infty$ depending on $\eps'$. 
To conclude, note that if $X \geq (2+\eps'/2)\log n$ then $S$ contains a clique of size at least $(2+\eps'/2)\log n$. 
\end{proof}

\subsection{Finding the planted clique} \label{sec:alg_finding} 

In order to find the planted clique, we start with the same step as in Section~\ref{sec:alg_detection} above: 
we sample all pairs of vertices among a set $S \subseteq [n]$ of size $m := \left( 2 + \eps' \right) (n/k) \log n$. 
As we show below, with probability $1-o(1)$ under $\p_{1}$, the set of vertices in the largest clique in $S$ is exactly the set of vertices in $S$ that are in the planted clique. 
Thus the remaining goal is to identify the vertices of the planted clique that are not in $S$. 

To do this, a natural idea is to query all pairs of vertices where one vertex is part of the largest clique in $S$ and the other vertex is not in $S$. 
Any vertex that is in the planted clique and not in $S$ will necessarily be connected to all planted clique vertices in $S$, while vertices not in $S$ and not in the planted clique will not be connected to all of the planted clique vertices in $S$ (with probability $1-o(1)$ under $\p_{1}$). 
Thus in the second step the algorithm selects all vertices not in $S$ that were connected to all vertices in $S$ where the pair was queried. 

Finally, the algorithm outputs the union of the two sets of vertices identified in the two steps. The following proof makes all this precise and proves Theorem~\ref{thm:estimation}\eqref{thm:estimation_ub}. Note that in the second step we take only a subset of the largest clique in $S$---this is done in order to lessen the number of queries made. 
Note furthermore that this algorithm has limited adaptivity, as it can be implemented in two ``rounds''.

\begin{proof}[Proof of Theorem~\ref{thm:estimation}\eqref{thm:estimation_ub}] 
The algorithm for finding the planted clique consists of two steps, the first step being the same as the one used for detection. 
\begin{itemize}
 \item \textbf{Step 1:} Choose a subset $S$ of the $n$ vertices of size $m := |S| = \left( 2 + \eps' \right) (n/k) \log n$, where $\eps'$ is chosen as in Section~\ref{sec:alg_detection}. We then query all pairs of vertices among $S$. 
 \item \textbf{Step 2:} Let $D$ be the set of vertices in the largest clique in $S$. Let $D' \subseteq D$ be 
a fixed 
subset of size $\left( 1 + \eps_{0} \right) \log n$ (e.g., take $D'$ to be the $\left( 1 + \eps_{0} \right) \log n$ nodes in $D$ with lowest label). (If $\eps_{0}$ is large such that $|D| \leq (1+\eps_{0}) \log n$, then let $D' := D$.) 
 We then query all pairs of vertices where one of the vertices is in $D'$ and the other is in $[n] \setminus S$. 
 
 Let $T$ denote the vertices in $[n] \setminus S$ that are connected to all vertices in $D'$. 
 The algorithm then outputs $D \cup T$ as its estimate for the planted clique. 
\end{itemize}
We have seen in Section~\ref{sec:alg_detection} that we make at most 
$(2+\eps_{0}) (n^{2} / k^{2}) \log^2 n$ queries in the first step, 
while in the second step we make at most 
$(1+\eps_{0}) n \log n$ queries. 
We now argue that this algorithm succeeds in finding the planted clique with probability $1-o(1)$.

As we argued in Section~\ref{sec:alg_detection}, 
we have that 
$\p_{1} \left( \left| D \right| \geq \left( 2 + \eps'/2 \right) \log n \right) = 1 - o(1)$ 
as $n \to \infty$. 
Furthermore, with probability $1 - o(1)$, the set $D$ contains only planted clique vertices. 
Indeed, as we discussed at the beginning of Section~\ref{sec:algorithms}, with probability $1 - o(1)$ the largest clique in an Erd\H{o}s-R\'enyi random graph with edge density $1/2$ has size at most $2\log n + 3$, so no vertex outside of the planted clique is in a clique of size greater than $2\log n + 3$. 
Thus in the first step of the algorithm we have found at least $\left( 2 + \eps'/2 \right) \log n$ vertices of the planted clique. Moreover, we have found all vertices of the planted clique that are in $S$. 

Any vertex in $[n] \setminus S$ that is in the planted clique will be connected to every planted clique vertex and hence every vertex in $D$. Thus all vertices of the planted clique are contained in $D \cup T$. 
To see that there are no false positives in this set, note that the probability that a vertex not in the planted clique is connected to a fixed set of $(1+\eps_{0}) \log n$ planted clique vertices is 
$2^{-(1+\eps_{0}) \log n} = n^{-(1+\eps_{0})}$. 
Taking a union bound over vertices in $[n] \setminus S$, we see that the probability that there exists a vertex not in the planted clique that is in $T$ is at most $n^{-\eps_{0}}$. 
\end{proof}

Note that this algorithm succeeds in finding the planted clique even though it does not check that all pairs of vertices within the planted clique are connected. 
In fact, it checks the edge between $O \left( k \log n + \log^2 n \right)$ pairs of vertices within the planted clique, instead of the $\Theta \left( k^2 \right)$ pairs that exist.

\section{Lower bounds} \label{sec:lb} 

To prove our lower bounds we introduce a simpler variant problem that removes all graph structure from the problem. 
In this hypothesis testing problem we consider the set $[n] = \left\{ 1, 2, \ldots, n \right\}$, 
where each element of the set is either marked or unmarked. 
Under the null hypothesis $\wt{H}_{0}$, all elements are unmarked. 
Under the alternative hypothesis $\wt{H}_{1}$, a uniformly randomly chosen subset $K \subseteq [n]$ of size $|K| = k$ is chosen and its elements are marked, and the elements of $[n] \setminus K$ are unmarked. 
We denote the two probability distributions over $\left\{ \text{unmarked, marked} \right\}^{[n]}$ by $\wt{\p}_{0}$ and $\wt{\p}_{1}$, respectively. 

We consider algorithms that can adaptively query $\left(i,j\right)$ pairs, where $1 \leq i < j \leq n$. 
We refer to such queries as \emph{pair queries} to distinguish them from the \emph{edge queries} of the original problem. 
When pair $(i,j)$ is queried, 
the query evaluates to true if both $i$ and $j$ are marked 
and it evaluates to false otherwise. 
The algorithm may ask $q$ such adaptive pair queries 
and use unbounded computational time to produce an output 
in $\{0,1\}$ (corresponding to $\wt{H}_{0}$ or $\wt{H}_{1}$). 
We again measure the performance of an algorithm $\wt{A}$ by its risk, defined as 
\[
 \wt{R}(\wt{A}) := \wt{\p}_{0} \left( \wt{A} = 1 \right) + \wt{\p}_{1} \left( \wt{A} = 0 \right).
\]
We consider randomized algorithms as well, in which case the type I and type II error probabilities in the display above are taken over the internal randomness of the algorithm as well. 

The following lemma connects this variant problem with the original hypothesis testing problem. 

\begin{lemma}[Reduction]\label{lem:reduction}
Suppose that there exists an algorithm $A$ that makes at most $q$ adaptive edge queries and achieves risk $R(A) \leq r$ for the hypothesis testing problem in~\eqref{eq:hypothesis}. 
Then there exists an algorithm $\wt{A}$ that makes at most $q$ adaptive pair queries in the variant problem described above---distinguishing between $\wt{H}_{0}$ and $\wt{H}_{1}$---and achieves risk $\wt{R}(\wt{A}) \leq r$. 
\end{lemma}
\begin{proof}
 There is a direct correspondence between the two hypothesis testing problems,  
 which allows the answers to pair queries to simulate answers to edge queries. 
 Specifically, marked elements of~$[n]$ correspond to planted clique vertices. 
 Thus a pair query that evaluates to true corresponds to querying two planted clique vertices, 
 while a pair query that evaluates to false corresponds to querying two vertices between which the edge is random. 
 Thus given the answer to a pair query, the answer to an edge query can be simulated as follows: 
 if the answer to the pair query is true, the answer to the corresponding edge query is that the edge exists, 
 while if the answer to the pair query is false, then flip a fair coin to answer the corresponding edge query. 
 
 Thus for any algorithm $A$ that makes at most $q$ adaptive edge queries, 
 there exists a corresponding algorithm $\wt{A}$ that makes at most $q$ adaptive pair queries in the variant problem 
 and simulates~$A$. We then let the output of $\wt{A}$ be the same as the output of the simulated algorithm~$A$. 
 Since the simulation of $A$ involves extra randomness, $\wt{A}$ is thus a randomized algorithm. 
 By conditioning on the extra randomness, it follows that the risk of $\wt{A}$ is the same as the risk of $A$. 
\end{proof}
This lemma implies that to prove Theorem~\ref{thm:detection}\eqref{thm:detection_lb} it suffices to prove the analogous result for the variant problem. Consequently, we turn our focus to the variant problem. 
Observe that under $\wt{H}_{0}$ all answers to all pair queries will be false. 
The next lemma considers the alternative hypothesis~$\wt{H}_{1}$. 

\begin{lemma}\label{lem:prob_two_marked_elements}
Let $q \leq \frac{n(n-1)}{k(k-1)}-1$. 
Let $\wt{A}$ be any algorithm that makes at most $q$ adaptive pair queries. 
Let $\cE_{q}$ denote the event that all of the 
pair queries of $\wt{A}$ evaluate to false. 
We then have that 
\begin{equation}\label{eq:prob_all_false}
 \wt{\p}_{1} \left( \cE_{q} \right) \geq 1 - q \frac{k(k-1)}{n(n-1)}. 
\end{equation}
In particular, if $q = o(n^{2}/k^{2})$ as $n \to \infty$, 
then $\wt{\p}_{1} \left( \cE_{q} \right) = 1 - o(1)$ as $n \to \infty$. 
\end{lemma}
\begin{proof}
To highlight the key elements of the proof, we first prove the claim for deterministic algorithms, where each query is a deterministic function of the previous queries and the  answers to them; at the end of the proof we address how the proof changes for randomized algorithms, which may use additional randomness. 
Thus for now assume that the algorithm $\wt{A}$ is deterministic. 
To describe the structure of deterministic algorithms we introduce some notation. 
We let $e_{1}, e_{2}, \ldots$ denote the pair queries made by the algorithm. 
Furthermore, let $X_{1}, X_{2}, \ldots$ denote the answers to the pair queries, as follows: 
$X_{\ell} = 0$ if the $\ell$th pair query $e_{\ell}$ evaluates to false, 
and $X_{\ell} = 1$ if the $\ell$th pair query $e_{\ell}$ evaluates to true. 
Any deterministic algorithm $\wt{A}$ can thus be described as follows: 
\begin{itemize}
\item First, $\wt{A}$ makes the pair query $e_{1} = \left( i_{1}, j_{1} \right)$. 
The algorithm receives the answer $X_{1}$ (which depends on the realization of $K$). 
\item The next pair query of $\wt{A}$ depends on the answer $X_{1}$: 
\begin{itemize}
\item if $X_{1} = 0$, then $\wt{A}$ makes the pair query $e_{2} = \left( i_{2}^{(0)}, j_{2}^{(0)} \right)$; 
\item if $X_{1} = 1$, then $\wt{A}$ makes the pair query $e_{2} = \left( i_{2}^{(1)}, j_{2}^{(1)} \right)$. 
\end{itemize}
The algorithm receives the answer $X_{2}$ (which again depends on the realization of $K$). 
\item The third pair query of $\wt{A}$ depends on the answers $X_{1}$ and $X_{2}$:
\begin{itemize} 
\item if $X_{1} = 0$ and $X_{2} = 0$, then $\wt{A}$ makes the pair query $e_{3} = \left( i_{3}^{(0,0)}, j_{3}^{(0,0)} \right)$; 
\item if $X_{1} = 0$ and $X_{2} = 1$, then $\wt{A}$ makes the pair query $e_{3} = \left( i_{3}^{(0,1)}, j_{3}^{(0,1)} \right)$; 
\item if $X_{1} = 1$ and $X_{2} = 0$, then $\wt{A}$ makes the pair query $e_{3} = \left( i_{3}^{(1,0)}, j_{3}^{(1,0)} \right)$; 
\item if $X_{1} = 1$ and $X_{2} = 1$, then $\wt{A}$ makes the pair query $e_{3} = \left( i_{3}^{(1,1)}, j_{3}^{(1,1)} \right)$. 
\end{itemize} 
The algorithm receives the answer $X_{3}$. 
\item And so on. The pair query $e_{\ell+1}$ of $\wt{A}$ depends on the answers $X_{1}, X_{2}, \ldots, X_{\ell}$ as follows: for every $x \in \left\{ 0, 1 \right\}^{\ell}$, 
if $\left( X_{1}, X_{2}, \ldots, X_{\ell} \right) = x$, 
then 
$e_{\ell + 1} = \left( i_{\ell+1}^{x}, j_{\ell+1}^{x} \right)$. 
\end{itemize}
Thus the set of pairs 
\begin{equation}\label{eq:pairs}
\left\{ \left( i_{1}, j_{1} \right) \right\} \cup 
\left\{ \left( i_{\ell}^{x}, j_{\ell}^{x} \right) : \ell \geq 2, x \in \left\{ 0, 1 \right\}^{\ell - 1} \right\}
\end{equation}
completely determines how the algorithm $\wt{A}$ behaves for any realization of $K$; 
and vice versa: any set of pairs as in~\eqref{eq:pairs} determines a deterministic pair query algorithm $\wt{A}$. (Note that in the description of a deterministic algorithm $\wt{A}$ we have only described how the algorithm makes the pair queries and not how the algorithm produces an output after making $q$ adaptive pair queries---for the purposes of the claim this is all that we care about.)

In the following we thus fix the deterministic algorithm $\wt{A}$ by fixing the set of pairs in~\eqref{eq:pairs}. 
Also, for notational convenience, 
we write $\left( i_{\ell}', j_{\ell}' \right)$ for $\left( i_{\ell}^{x}, j_{\ell}^{x} \right)$ 
when $\ell \geq 2$ and $x = \left( 0, 0, \ldots, 0 \right) \in \left\{ 0, 1 \right\}^{\ell - 1}$; 
furthermore, let $i_{1}' \equiv i_{1}$ and $j_{1}' \equiv j_{1}$. 
Recall that  $\cE_{q}$ denotes the event that the first $q$ adaptive pair queries of $\wt{A}$ all evaluate to false. 
We prove~\eqref{eq:prob_all_false}  
by determining the conditional law of $K$ given the event $\cE_{q}$. 
Note that we fixed the set of pairs in~\eqref{eq:pairs}, 
and thus we know that, given $\cE_{q}$, the first $q$ pair queries were 
$e_{1} = \left(i_{1}', j_{1}' \right), \ldots, e_{q} = \left(i_{q}', j_{q}' \right)$. 
Let $\cS_{q}$ denote the set of $k$-tuples such that there was a pair query among these first $q$ pair queries that queried a pair from this $k$-tuple; 
that is, 
\[
\cS_{q} := \left\{ S \subseteq [n] : \left| S \right| = k, \exists\ \ell \in [q] : i_{\ell}', j_{\ell}' \in S \right\}.
\]
Since all pair queries evaluated to false, no $k$-tuple in $\cS_{q}$ can be the marked subset given $\cE_{q}$ (since otherwise a pair query would have evaluated to true). 
That is, 
\[
\wt{\p}_{1} \left( K = S \, \middle| \, \cE_{q} \right) = 0 
\]
for all $S \in \cS_{q}$. 
Now let us consider a $k$-tuple that is not in $\cS_{q}$. 
By Bayes's rule we have that 
\[
\wt{\p}_{1} \left( K = S \, \middle| \, \cE_{q} \right)
= 
\frac{\wt{\p}_{1} \left( \cE_{q} \, \middle| \, K = S \right) \wt{\p}_{1} \left( K = S \right)}{\wt{\p}_{1} \left( \cE_{q} \right)}.
\]
Since the prior on $K$ is uniform, 
we have that 
$\wt{\p}_{1} \left( K = S \right) = 1 / \binom{n}{k}$. 
Now since $S \notin \cS_{q}$, 
if $K = S$, then the answers to the pair queries 
$e_{1} = \left(i_{1}', j_{1}' \right), \ldots, e_{q} = \left(i_{q}', j_{q}' \right)$ 
are necessarily all false (due to the definition of $\cS_{q}$). 
Therefore for every $S \notin \cS_{q}$ we have that 
$\wt{\p}_{1} \left( \cE_{q} \, \middle| \, K = S \right) = 1$. 
Thus we have shown that for every $S \notin \cS_{q}$ we have that 
\[
\wt{\p}_{1} \left( K = S \, \middle| \, \cE_{q} \right)
= \frac{1}{\binom{n}{k} \wt{\p}_{1} \left( \cE_{q} \right)}.
\]
Altogether, we have shown that the conditional law of $K$ given $\cE_{q}$ is given by 
\[
\wt{\p}_{1} \left( K = S \, \middle| \, \cE_{q} \right)
= \frac{1}{\binom{n}{k} \wt{\p}_{1} \left( \cE_{q} \right)} \mathbf{1}_{\left\{ S \notin \cS_{q} \right\}}.
\]
Notice that this conditional probability is equal for all $k$-tuples that are not in $\cS_{q}$. 
Therefore we also have that 
\[
\wt{\p}_{1} \left( K = S \, \middle| \, \cE_{q} \right)
= \frac{1}{\binom{n}{k} - \left| \cS_{q} \right|} \mathbf{1}_{\left\{ S \notin \cS_{q} \right\}}.
\]
Putting together the previous two displays we have that 
\begin{equation}\label{eq:probEq}
\wt{\p}_{1} \left( \cE_{q} \right) 
= 1 - \frac{\left| \cS_{q} \right|}{\binom{n}{k}}.
\end{equation}
Note that every pair $(i,j)$ (where $1 \leq i < j \leq n$) is part of 
exactly $\binom{n-2}{k-2}$ different subsets of $k$ elements. 
This implies the following upper bound on the size of $\cS_{q}$: 
\[
\left| \cS_{q} \right| 
\leq q \binom{n-2}{k-2}
= q \frac{k(k-1)}{n(n-1)} \binom{n}{k}. 
\]
Plugging this bound back into~\eqref{eq:probEq} we obtain~\eqref{eq:prob_all_false}, as desired.

Finally, we discuss randomized algorithms, which may use additional randomness. 
We may condition on the extra randomness and then use the argument above for deterministic algorithms. 
This shows that no matter what the realization of the additional randomness is, the conditional probability of all $q$ pair queries evaluating to false is at least 
$1 - q \frac{k(k-1)}{n(n-1)}$. 
Taking an expectation over the additional randomness then shows the desired claim. 
\end{proof}

We are now ready to prove the analogue of Theorem~\ref{thm:detection}\eqref{thm:detection_lb} for the variant problem. 
\begin{corollary}[Detecting a marked set of elements]\label{cor:detection_lb_variant}
Consider the hypothesis testing problem $\wt{H}_{0}$ versus $\wt{H}_{1}$. 
Let $q = o(n^{2}/k^{2})$ as $n \to \infty$. 
If an algorithm $\wt{A}$ makes at most $q$ adaptive pair queries 
then its risk must satisfy $\wt{R}(\wt{A}) \geq 1 - o(1)$ as $n \to \infty$. 
\end{corollary}
\begin{proof}
No matter what algorithm $\wt{A}$ does, 
all of its pair queries will evaluate to false under $\wt{H}_{0}$ (by definition), 
and all of its pair queries will evaluate to false under $\wt{H}_{1}$ with probability $1-o(1)$ (by Lemma~\ref{lem:prob_two_marked_elements}). 
Suppose that $\wt{A}$ outputs $0$ with probability $\alpha$ and $1$ with probability $1-\alpha$ when all of its queries evaluate to false, where $\alpha \in [0,1]$. 
The first sentence of the proof then implies that its risk is at least 
$\wt{R}(\wt{A}) \geq (1-\alpha) + \alpha(1-o(1)) = 1 - o(1)$. 
\end{proof}

\begin{proof}[Proof of Theorem~\ref{thm:detection}\eqref{thm:detection_lb}]
 This follows directly from Lemma~\ref{lem:reduction} and Corollary~\ref{cor:detection_lb_variant}. 
\end{proof}

We now turn to proving Theorem~\ref{thm:estimation}\eqref{thm:estimation_lb}. 
Here too we leverage the connection to the corresponding estimation problem for the simplified variant problem, 
where we aim to estimate the set of marked elements.

\begin{lemma}\label{lem:estimation_lb_variant}
 Let $K \subseteq [n]$ be a uniformly randomly chosen set of size $|K| = k$, where $1 \leq k < n$. Let the elements of $K$ be marked and let the elements of $[n] \setminus K$ be unmarked. 
 Let $q = o(n^{2} / k^{2} + n)$ as $n \to \infty$. 
 If $\wh{K}$ is any estimator of the marked set $K$ that is based on at most $q$ adaptive pair queries, 
 then $\wh{K}$ satistifies 
 $\wt{\p}_{1} \left( \wh{K} = K \right) = o(1)$ as $n \to \infty$. 
\end{lemma}
\begin{proof} There are two cases to consider. 
 First, consider the case when $q = o(n^{2}/k^{2})$. 
 The proof of Lemma~\ref{lem:prob_two_marked_elements} shows that, 
 with probability $1-o(1)$,  
 after $q$ adaptive pair queries there remain a $(1-o(1))$ fraction of subsets of size $k$ that are equally likely to be the marked set. 
 No estimator can do better than pick randomly among these, 
 and this will succeed with probability 
 $\left( 1 + o(1) \right) / \binom{n}{k} = o(1)$. 
 
Next, consider the case when $q = o(n)$. 
In this case we show that it is not possible to estimate the marked set even for algorithms with significantly more information. 
Specifically, we consider algorithms that can adaptively query $(i,j)$ pairs, where $1 \leq i < j \leq n$, and when pair $(i,j)$ is queried, the algorithm learns, for both $i$ and $j$, whether they are marked or unmarked. We refer to such queries as \emph{strong pair queries} to distinguish them from \emph{pair queries}. 
From the answer to a strong pair query it is possible to determine the answer to the appropriate pair query. Therefore any algorithm that makes $q$ adaptive pair queries can be simulated by an algorithm that makes~$q$ adaptive strong pair queries. 
Thus in order to prove the claim it suffices to show that 
if $\wh{K}$ is any estimator of the marked set $K$ that is based on at most~$q$ adaptive strong pair queries, 
then~$\wh{K}$ satistifies 
$\wt{\p}_{1} \left( \wh{K} = K \right) = o(1)$ as $n \to \infty$. 
This is what we will show; thus in the following we consider algorithms that make at most $q$ adaptive strong pair queries, and we assume that $q = o(n)$.

We now argue that for $q < n/2$ we may, without loss of generality, consider the algorithm that makes the strong pair queries 
$(1,2), (3,4), \ldots, (2q-1,2q)$. 
We argue this by induction. Since the marked set $K$ is chosen uniformly at random, 
the first strong pair query may be $(1,2)$ without loss of generality. 
There are now three cases to consider, depending on the answer to this first strong pair query. 
\begin{itemize}
\item \textbf{Both elements are unmarked.} Suppose that the answer to the strong pair query $(1,2)$ is that both $1$ and $2$ are unmarked. Thus $1,2 \notin K$ and therefore neither $1$ nor $2$ will be in any estimator $\wh{K}$ (since otherwise the estimator will be incorrect). 
The algorithm thus knows that $K \subseteq [n] \setminus \{1,2\}$. Moreover, by Bayes's rule, the conditional distribution of $K$, given that both $1$ and $2$ are unmarked, is uniform among $k$-tuples in $[n] \setminus \{1,2\}$. 
\item \textbf{One element is unmarked, the other is marked.} 
Suppose that the answer to the strong pair query $(1,2)$ is that $1$ is marked and $2$ is unmarked. Thus $1 \in K$ and therefore $1$ will be in any estimator $\wh{K}$ (since otherwise the estimator will be incorrect). We also learn that $2 \notin K$, so~$2$ will not be in any estimator $\wh{K}$ (since otherwise the estimator will be incorrect). 
Moreover, by Bayes's rule, the conditional distribution of $K$, given that $1$ is marked and $2$ is unmarked, is uniform among $k$-tuples in $[n]$ that contain the element $1$ and do not contain the element~$2$. Thus the conditional distribution of $K \setminus \{1,2\}$, given that $1$ is marked and $2$ is unmarked, is uniform among $(k-1)$-tuples in $[n] \setminus \{1,2\}$. 

The case where $1$ is unmarked and $2$ is marked is analogous. 

\item \textbf{Both elements are marked.} 
Suppose that the answer to the strong pair query $(1,2)$ is that both $1$ and $2$ are marked. 
Thus $1,2 \in K$ and therefore both $1$ and $2$ will be in any estimator~$\wh{K}$ (otherwise the estimator will be incorrect). 
Moreover, by Bayes's rule, the conditional distribution of $K$, given that both $1$ and $2$ are marked, is uniform among $k$-tuples in $[n]$ that contain both $1$ and $2$. 
Thus the conditional distribution of $K \setminus \{1,2\}$, given that both $1$ and $2$ are marked, is uniform among $(k-2)$-tuples in $[n] \setminus \{1, 2\}$.
\end{itemize}
In summary, no matter what the answer to the strong pair query $(1,2)$ is, the algorithm  deduces the following two points. 
\begin{itemize}
\item For elements $1$ and $2$, the algorithm knows whether or not to include them in any estimator~$\wh{K}$ that has any possibility of being correct. 
\item The conditional distribution of $K \setminus \{1,2\}$, given the answer to the strong pair query $(1,2)$, is uniform among $m$-tuples in $[n] \setminus \{1,2\}$; 
here $m = k$ if both $1$ and $2$ are unmarked, 
$m = k-1$ if one is unmarked and the other is marked, 
and $m = k-2$ if both $1$ and $2$ are marked. 
\end{itemize}
Due to the \emph{uniformity} of the conditional distribution in the last bullet point, 
the next strong pair query may, without loss of generality, be $(3,4)$. 
More generally, after having made the strong pair queries 
$(1,2), (3,4), \ldots, (2\ell-1,2\ell)$, 
the algorithm deduces the following two points. 
\begin{itemize}
\item For each element in  $\{1,2, \ldots, 2\ell \}$, the algorithm knows whether or not to include them in any estimator $\wh{K}$ that has any possibility of being correct. 
\item The conditional distribution of $K \setminus \{ 1, 2, \ldots 2\ell \}$, 
given the answers to the strong pair queries 
$(1,2), (3,4), \ldots, (2\ell-1,2\ell)$, 
is uniform among $m$-tuples in $[n] \setminus \{1, 2, \ldots, 2 \ell \}$, 
where $m$ is equal to $k$ minus the number of marked elements in $\{ 1, 2, \ldots, 2 \ell \}$. 
\end{itemize}
Again, due to the uniformity of the conditional distibution in the bullet point above, 
the next strong pair query may, without loss of generality, be $(2\ell+1, 2\ell+2)$. 
This finishes the proof of the induction.

Finally, we analyze the algorithm that makes the strong pair queries 
$(1,2), (3,4), \ldots, (2q-1,2q)$. 
After the answers to these strong pair queries, the algorithm knows for each element in $[2q] = \{1, 2, \ldots, 2q \}$ whether they are marked or unmarked. 
Let $S$ denote the subset of marked elements in $[2q]$, 
and let $X := |S|$. 
Any estimator $\wh{K}$ that has any possibility of being correct must include~$S$ as a subset (since otherwise the estimator will be incorrect); similarly, any estimator $\wh{K}$ that has any possibility of being correct must not include any elements in  $[2q] \setminus S$. 
If $X = k$, then this determines that the estimator should be $\wh{K} = S$, and indeed in this case the estimator is correct: $\wh{K} = K$. 
If $X < k$, then the estimator has to choose a subset $\wh{T} \subseteq [n] \setminus [2q]$ of size $k-X$ and outputs $\wh{K} = S \cup \wh{T}$. 
The estimator is then correct (that is, $\wh{K} = K$ holds) if and only if $\wh{T} = K \setminus [2q]$. 
As we have argued above, 
the conditional distribution of 
$K \setminus [2q]$, given the answers to the strong pair queries 
$(1,2), (3,4), \ldots, (2q-1,2q)$, 
is uniform among $(k-X)$-tuples in $[n] \setminus [2q]$. 
Due to the uniformity of this conditional distribution, 
for \emph{any} estimator $\wh{T}$ 
the conditional probability of $\wh{T} = K \setminus [2q]$ is equal to $1/\binom{n-2q}{k-X}$. 
Putting everything together we have thus obtained that 
\begin{equation}\label{eq:estimation}
\wt{\p}_{1} \left( \wh{K} = K \right) 
= \wt{\E}_{1} \left[ \wt{\p}_{1} \left( \wh{K} = K \, \middle| \, X \right) \right] 
= \wt{\E}_{1} \left[ \wt{\p}_{1} \left( \wh{T} = K \setminus [2q] \, \middle| \, X \right) \right] 
= \wt{\E}_{1} \left[ \frac{1}{\binom{n-2q}{k-X}} \right].
\end{equation}
Since the distribution of $K$ is uniform among $k$-tuples in $[n]$, 
the distribution of $X$ is hypergeometric with parameters $n$, $k$, and $2q$. 
We now distinguish three cases based on how the parameters $n$, $k$, and $2q$ relate to each other, and in each case we bound the expected value in~\eqref{eq:estimation}. 
\begin{itemize}
\item \textit{Case 1: $k \leq 2q$.} 
Since $q = o(n)$, we have that $q < n/4$ for all $n$ large enough. 
Thus for all $n$ large enough we have that $k - X \leq k \leq 2q < n - 2q$. 
This implies that, for all $n$ large enough, if $X < k$, then $\binom{n-2q}{k-X} \geq n - 2q \geq n/2$. Also, if $X = k$, then $\binom{n-2q}{k-X} = 1$. We thus have, for all $n$ large enough, that 
\begin{equation}\label{eq:case1_bound}
\wt{\E}_{1} \left[ \frac{1}{\binom{n-2q}{k-X}} \right] 
\leq \frac{2}{n} + \wt{\p}_{1} \left( X = k \right).
\end{equation}
Since $X$ is a hypergeometric random variable with parameters $n$, $k$, and $2q$, we have that 
\begin{equation}\label{eq:case1_hypergeometric}
\wt{\p}_{1} \left( X = k \right) 
= \frac{\binom{k}{k} \binom{n-k}{2q-k}}{\binom{n}{2q}} 
= \frac{(2q)(2q-1) \cdot \ldots \cdot (2q-k+1)}{n(n-1) \cdot \ldots \cdot (n-k+1)}
\leq \frac{2q}{n}.
\end{equation}
Combining~\eqref{eq:case1_bound} and~\eqref{eq:case1_hypergeometric} we have that 
$\wt{\E}_{1} \left[ 1/\binom{n-2q}{k-X} \right] 
\leq (2q+2)/n = o(1)$ 
as $n \to \infty$. 

\item \textit{Case 2: $2q < k < n - 2q$.} Since $X \leq 2q$, in this case we always have that $k-X > 2q-X \geq 0$, so $k - X \geq 1$. Also, $k - X \leq k < n - 2q$. Put together, we have that $1 \leq k - X < n - 2q$, which implies that $\binom{n-2q}{k-X} \geq n - 2q$. Therefore in this case we have that 
\[
\wt{\E}_{1} \left[ \frac{1}{\binom{n-2q}{k-X}} \right] 
\leq \frac{1}{n-2q}.
\]

\item \textit{Case 3: $n-2q \leq k$.} 
Since $q = o(n)$, we have that $q < n/4$ for all $n$ large enough. Note also that $X \leq 2q$. 
Thus for all $n$ large enough we have that 
$k - X \geq (n-2q) - 2q > 0$, so $k - X \geq 1$. Also note that $k - X \leq n - 2q$ by definition. 
If $1 \leq k - X < n - 2q$, then 
$\binom{n-2q}{k-X} \geq n - 2q \geq n/2$, where the second inequality holds for all $n$ large enough. 
We thus have, for all $n$ large enough, that 
\begin{equation}\label{eq:case3_bound}
\wt{\E}_{1} \left[ \frac{1}{\binom{n-2q}{k-X}} \right] 
\leq \frac{2}{n} + \wt{\p}_{1} \left( X = k - (n-2q) \right).
\end{equation}
Since $X$ is a hypergeometric random variable with parameters $n$, $k$, and $2q$, we have that 
\begin{equation}\label{eq:case3_hypergeometric}
\wt{\p}_{1} \left( X = k - (n-2q) \right) 
= \frac{\binom{k}{k-(n-2q)} \binom{n-k}{n-k}}{\binom{n}{2q}} 
= \frac{2q(2q-1) \cdot \ldots \cdot (k-(n-2q)+1)}{n(n-1) \cdot \ldots \cdot (k+1)} 
\leq \frac{2q}{n}.
\end{equation}
Combining~\eqref{eq:case3_bound} and~\eqref{eq:case3_hypergeometric} we have that 
$\wt{\E}_{1} \left[ 1/\binom{n-2q}{k-X} \right] 
\leq (2q+2)/n = o(1)$ 
as $n \to \infty$. 
\end{itemize}
In summary, in all three cases above we have that 
$\wt{\E}_{1} \left[ 1/\binom{n-2q}{k-X} \right] 
= o(1)$
as $n \to \infty$. 
Combining this with~\eqref{eq:estimation} proves the claim. 
\end{proof}

\begin{proof}[Proof of Theorem~\ref{thm:estimation}\eqref{thm:estimation_lb}]
 There exists a direct correspondence between the problem of estimating the planted clique and the problem of estimating the marked set in the variant problem. This correspondence for the estimation problem is analogous to the correspondence for the detection problem described in the proof of Lemma~\ref{lem:reduction}. 
 The proof then follows directly from Lemma~\ref{lem:estimation_lb_variant} and this correspondence. 
\end{proof}


\section*{Acknowledgements}
 
M.Z.R.\ is grateful to David Gamarnik for helpful discussions. 
We also thank Jay Mardia, Joe Neeman, an anonymous reviewer, and an anonymous associate editor for helpful comments, feedback, and questions which helped improve the manuscript. 


\bibliographystyle{abbrv}
\bibliography{bib}

\begin{thebibliography}{10}

\bibitem{AKS98}
N.~Alon, M.~Krivelevich, and B.~Sudakov.
\newblock Finding a large hidden clique in a random graph.
\newblock {\em Random Structures \& Algorithms}, 13(3-4):457--466, 1998.

\bibitem{alweiss2019subgraph}
R.~Alweiss, C.~Ben~Hamida, X.~He, and A.~Moreira.
\newblock On the subgraph query problem.
\newblock Preprint available at \url{https://arxiv.org/abs/1911.04413}, 2019.

\bibitem{anagnostopoulos2016community}
A.~Anagnostopoulos, J.~{\L}{a}cki, S.~Lattanzi, S.~Leonardi, and M.~Mahdian.
\newblock Community detection on evolving graphs.
\newblock In {\em Advances in Neural Inf. Proc. Systems}, pages 3522--3530,
  2016.

\bibitem{finance}
S.~Arora, B.~Barak, M.~Brunnermeier, and R.~Ge.
\newblock Computational complexity and information asymmetry in financial
  products.
\newblock {\em Communications of the ACM}, 54(5):101--107, 2011.

\bibitem{berthet2013complexity}
Q.~Berthet and P.~Rigollet.
\newblock {Complexity Theoretic Lower Bounds for Sparse Principal Component
  Detection}.
\newblock In {\em Proceedings of the 26th Annual Conference on Learning Theory
  (COLT)}, pages 1046--1066, 2013.

\bibitem{berthet2013optimal}
Q.~Berthet and P.~Rigollet.
\newblock Optimal detection of sparse principal components in high dimension.
\newblock {\em The Annals of Statistics}, 41(4):1780--1815, 2013.

\bibitem{bollobas1976cliques}
B.~Bollob{\'a}s and P.~Erd{\H{o}}s.
\newblock Cliques in random graphs.
\newblock {\em Mathematical Proceedings of the Cambridge Philosophical
  Society}, 80(3):419--427, 1976.

\bibitem{brennan2019optimal}
M.~Brennan and G.~Bresler.
\newblock {Optimal Average-Case Reductions to Sparse PCA: From Weak Assumptions
  to Strong Hardness}.
\newblock Preprint at \url{https://arxiv.org/abs/1902.07380}, 2019.

\bibitem{brennan2018reducibility}
M.~Brennan, G.~Bresler, and W.~Huleihel.
\newblock Reducibility and computational lower bounds for problems with planted
  sparse structure.
\newblock In {\em Proceedings of the 31st Annual Conference on Learning Theory
  (COLT)}, pages 48--166, 2018.

\bibitem{conlon2019online}
D.~Conlon, J.~Fox, A.~Grinshpun, and X.~He.
\newblock {Online Ramsey Numbers and the Subgraph Query Problem}.
\newblock In {\em {Building Bridges II}}, pages 159--194. Springer, 2019.

\bibitem{DGGP14}
Y.~Dekel, O.~Gurel-Gurevich, and Y.~Peres.
\newblock Finding hidden cliques in linear time with high probability.
\newblock {\em Combinatorics, Probability and Computing}, 23(1):29--49, 2014.

\bibitem{DM15}
Y.~Deshpande and A.~Montanari.
\newblock {Finding Hidden Cliques of Size $\sqrt{N/e}$ in Nearly Linear Time}.
\newblock {\em Foundations of Computational Mathematics}, 15(4):1069--1128,
  2015.

\bibitem{FGNRT18}
U.~Feige, D.~Gamarnik, J.~Neeman, M.~Z. R{\'a}cz, and P.~Tetali.
\newblock {Finding cliques using few probes}.
\newblock {\em Random Structures \& Algorithms}, 56(1):142--153, 2020.

\bibitem{feige2010finding}
U.~Feige and D.~Ron.
\newblock Finding hidden cliques in linear time.
\newblock In {\em Proceedings of the 21st International Meeting on
  Probabilistic, Combinatorial and Asymptotic Methods for the Analysis of
  Algorithms}, pages 189--204. Discrete Mathematics and Theoretical Computer
  Science, 2010.

\bibitem{feldman2017statistical}
V.~Feldman, E.~Grigorescu, L.~Reyzin, S.~S. Vempala, and Y.~Xiao.
\newblock Statistical algorithms and a lower bound for detecting planted
  cliques.
\newblock {\em Journal of the ACM (JACM)}, 64(2):1--37, 2017.

\bibitem{ferber2016finding}
A.~Ferber, M.~Krivelevich, B.~Sudakov, and P.~Vieira.
\newblock Finding {H}amilton cycles in random graphs with few queries.
\newblock {\em Random Structures \& Algorithms}, 49(4):635--668, 2016.

\bibitem{ferber2017finding}
A.~Ferber, M.~Krivelevich, B.~Sudakov, and P.~Vieira.
\newblock Finding paths in sparse random graphs requires many queries.
\newblock {\em Random Structures \& Algorithms}, 50(1):71--85, 2017.

\bibitem{gao2017sparse}
C.~Gao, Z.~Ma, and H.~H. Zhou.
\newblock {Sparse CCA: Adaptive estimation and computational barriers}.
\newblock {\em The Annals of Statistics}, 45(5):2074--2101, 2017.

\bibitem{hartmann2016clustering}
T.~Hartmann, A.~Kappes, and D.~Wagner.
\newblock Clustering evolving networks.
\newblock In {\em Algorithm Engineering}, pages 280--329. Springer, 2016.

\bibitem{jerrum1992large}
M.~Jerrum.
\newblock Large cliques elude the {M}etropolis process.
\newblock {\em Random Structures \& Algorithms}, 3(4):347--359, 1992.

\bibitem{crypto}
A.~Juels and M.~Peinado.
\newblock {Hiding Cliques for Cryptographic Security}.
\newblock {\em Designs, Codes, and Cryptography}, 20(3):269--280, 2000.

\bibitem{kearns1998efficient}
M.~Kearns.
\newblock Efficient noise-tolerant learning from statistical queries.
\newblock {\em Journal of the ACM (JACM)}, 45(6):983--1006, 1998.

\bibitem{kuvcera1995expected}
L.~Ku{\v{c}}era.
\newblock Expected complexity of graph partitioning problems.
\newblock {\em Discrete Applied Mathematics}, 57(2-3):193--212, 1995.

\bibitem{lugosi17}
G.~Lugosi.
\newblock {Lectures on Combinatorial Statistics}.
\newblock Available at \url{http://www.econ.upf.edu/~lugosi/SaintFlour.pdf},
  2017.

\bibitem{mardia2020finding}
J.~Mardia, H.~Asi, and K.~A. Chandrasekher.
\newblock {Finding Planted Cliques in Sublinear Time}.
\newblock Preprint available at \url{https://arxiv.org/abs/2004.12002}, 2020.

\bibitem{matula1972}
D.~W. Matula.
\newblock {The Employee Party Problem}.
\newblock In {\em Notices of the American Mathematical Society}, volume~19,
  pages A--382, 1972.

\bibitem{mazumdar2017clustering}
A.~Mazumdar and B.~Saha.
\newblock Clustering with noisy queries.
\newblock In {\em Advances in Neural Information Processing Systems}, pages
  5788--5799, 2017.

\bibitem{mazumdar2017query}
A.~Mazumdar and B.~Saha.
\newblock Query complexity of clustering with side information.
\newblock In {\em Advances in Neural Information Processing Systems}, pages
  4682--4693, 2017.

\bibitem{bio}
R.~Milo, S.~Shen-Orr, S.~Itzkovitz, N.~Kashtan, D.~Chklovskii, and U.~Alon.
\newblock Network motifs: Simple building blocks of complex networks.
\newblock {\em Science}, 298(5594):824--827, 2002.

\bibitem{vinayak2016crowdsourced}
R.~K. Vinayak and B.~Hassibi.
\newblock Crowdsourced clustering: Querying edges vs triangles.
\newblock In {\em Advances in Neural Information Processing Systems}, pages
  1316--1324, 2016.

\end{thebibliography}




\end{document}